\documentclass[12pt,reqno]{amsart}

\usepackage{amssymb,amsmath}
\usepackage{amsfonts}
\usepackage{amsthm}
\usepackage{mathrsfs}
\usepackage{graphicx}
\usepackage[T1]{fontenc}% Brzd\k{e}k

\theoremstyle{plain} 
\newtheorem{theorem}{Theorem}

\newtheorem{lemma}[theorem]{Lemma}

\theoremstyle{definition} 
\newtheorem{definition}[theorem]{Definition}

\newtheorem{example}[theorem]{Example}
\newcommand{\R}{\ensuremath{\mathbb{R}}}

\newcommand{\T}{\ensuremath{\mathbb{T}}}
\newcommand{\N}{\ensuremath{\mathbb{N}}}
\newcommand{\C}{\ensuremath{\mathbb{C}}}

\numberwithin{equation}{section}
\numberwithin{theorem}{section}

\small\normalsize

\begin{document}

\title[Hyers--Ulam stability for quantum equations]{Hyers--Ulam stability for quantum equations}
\author[Anderson]{Douglas R. Anderson} 
\address{Department of Mathematics \\
         Concordia College \\
         Moorhead, MN 56562 USA}
\email{andersod@cord.edu}
\author[Onitsuka]{Masakazu Onitsuka}
\address{Okayama University of Science \\
Department of Applied Mathematics \\
Okayama, 700-0005, Japan}
\email{onitsuka@xmath.ous.ac.jp}


\begin{abstract}
We introduce and study the Hyers--Ulam stability (HUS) of a Cayley quantum ($q$-difference) equation of first order, where the constant coefficient is allowed to range over the complex numbers. In particular, if this coefficient is non-zero, then the quantum equation has Hyers--Ulam stability for certain values of the Cayley parameter, and we establish the best (minimal) HUS constant in terms of the coefficient only, independent of $q$ and the Cayley parameter. If the Cayley parameter equals one half, then there is no Hyers--Ulam stability for any coefficient value in the complex plane.
\end{abstract}

\maketitle\thispagestyle{empty}

%%%%%%%%%%%%%%%%% 
%               % 
% SECTION Intro % 
%               % 
%%%%%%%%%%%%%%%%%

\section{introduction} 

There has been much recent interest in Hyers--Ulam stability (HUS) in relation to $h$-difference equations, differential equations, and dynamic equations on time scales. For example, see \cite{and}-\cite{bos} and \cite{jung}-\cite{satco}. However, very little work has been done for quantum equations, also known as $q$-difference equations \cite{kac}. To fill this gap, we explore the notion of HUS for first-order quantum equations with a complex coefficient and a Cayley parameter, for Cayley equations introduced in another context in \cite{cieslinski}. In Section 2, we introduce the first-order quantum Cayley equation with complex coefficient, where the Cayley parameter ranges in $[0,\frac{1}{2})$, and explore the stability of this equation. We find that when the equation has HUS, the best HUS constant can be found precisely in terms of the reciprocal modulus of the complex coefficient. In Section 3, we will see that the stability changes drastically for the Cayley value of $\frac{1}{2}$, and provide a few examples of how the stability breaks down.

%%%%%%%%%%%%%%%%% 
%               % 
% SECTION 2     % 
%               % 
%%%%%%%%%%%%%%%%%

\section{HUS for quantum equations with complex constant coefficient}

Let $q>1$, and define the quantum set $\T=q^{\N_0}:=\{1,q,q^2,q^3,\cdots\}$, where $\N$ is the set of natural numbers, and $\N_0:=\N\cup\{0\}$.
Let $\eta\in\left[0,\frac{1}{2}\right)$. Throughout this work, let $w\in\C$ satisfy the condition
\begin{equation}\label{wrestriction}
 w\in\C\backslash\left\{\frac{-1}{(1-\eta)(q-1)q^k},\; \frac{1}{\eta(q-1)q^k}\right\}_{k=0}^{\infty}. 
\end{equation} 
For $w$ satisfying \eqref{wrestriction}, $x:\T\rightarrow\C$, and $\eta\in\left[0,\frac{1}{2}\right)$, we consider on $\T$ the Hyers--Ulam stability of the first-order linear homogeneous Cayley quantum equation with complex coefficient given for $t\in\T$ by
\begin{equation}\label{maineq}
 D_q x(t) - w \left\langle x(t) \right\rangle_{\eta} = 0, \quad D_q x(t):=\frac{x(qt)-x(t)}{(q-1)t}, \quad \langle x(t) \rangle_{\eta}:=\eta x(qt)+(1-\eta)x(t).
\end{equation}
Note that \eqref{maineq} is the Cayley equation introduced in \cite{cieslinski} in a different context.

% Lemma %

\begin{lemma}\label{GeneSol}
For any $w$ satisfying \eqref{wrestriction}, the general solution of \eqref{maineq} is given by
\begin{equation}\label{generalsolformx}
  x(t) = c\prod_{k=0}^{\log_q t-1} \frac{1+w(1-\eta)(q-1)q^k}{1-w\eta(q-1)q^k},
\end{equation} 
where $c\in\C$ is an arbitrary constant.
\end{lemma}

\begin{proof}
Let $w$ satisfy \eqref{wrestriction}. Working directly from \eqref{maineq}, if we rewrite it to interpret it as a recurrence relation, then
\[ x(qt) = \left(\frac{1+w(1-\eta)(q-1)t}{1-w\eta(q-1)t}\right)x(t), \quad t\in\T=q^{\N_0}, \quad q>1. \]
Solving this for $x(t)$, we arrive at \eqref{generalsolformx}, which is a well-defined function on $\T$ for $w$ satisfying \eqref{wrestriction}. Substituting form \eqref{generalsolformx} back into \eqref{maineq} as a check,
\begin{eqnarray*}
 \lefteqn{D_q x(t) - w \left\langle x(t) \right\rangle_{\eta}} \\
  &=& \frac{x(qt)-x(t)}{(q-1)t} - w\left[\eta x(qt)+(1-\eta)x(t)\right] \\
  &=& \left[\left(\frac{1}{(q-1)t}-w\eta\right)\left(\frac{1+w(1-\eta)(q-1)t}{1-w\eta(q-1)t}\right) - \left(\frac{1}{(q-1)t}+w(1-\eta)\right) \right] x(t) \\
	&=& \left[\left(\frac{1-w\eta(q-1)t}{(q-1)t}\right)\left(\frac{1+w(1-\eta)(q-1)t}{1-w\eta(q-1)t}\right) - \left(\frac{1+w(1-\eta)(q-1)t}{(q-1)t}\right) \right] x(t) \\
	&=& 0.
\end{eqnarray*}
This establishes the result and ends the proof.
\end{proof}

% Definition %

\begin{definition}
The quantum equation \eqref{maineq} has Hyers--Ulam stability (HUS) if and only if there exists a constant $K>0$ with the following property:
\begin{quote}
For an arbitrary $\varepsilon>0$, if a function $\phi:\T\rightarrow\C$ satisfies 
\begin{equation}\label{phiepineq}
 \left| D_q\phi(t) - w \left\langle \phi(t)\right\rangle_{\eta}\right| \le \varepsilon
\end{equation} 
for all $t\in\T$, then there exists a solution $x:\T\rightarrow\C$ of \eqref{maineq} such that 
$$|\phi(t)-x(t)|\le K\varepsilon$$ 
for all $t\in\T$. 
\end{quote}
Such a constant $K$ is called an HUS constant for \eqref{maineq} on $\T$.
\end{definition}

% Lemma: Variation of Parameters %

\begin{lemma}\label{VofP}
Fix $q>1$. Let $w$ satisfy \eqref{wrestriction}, and let $\eta\in\left[0,\frac{1}{2}\right)$. For an arbitrary $\varepsilon>0$, if a function $\phi:\T\rightarrow\C$ satisfies the inequality \eqref{phiepineq}, then for $t\in\T$, $\phi$ has the form $\phi:=PS+cP$, where
\begin{equation}\label{phiform} 
   P(t)=\prod_{k=0}^{\log_q t-1} \frac{1+w(1-\eta)(q-1)q^k}{1-w\eta(q-1)q^k}, 
	 \quad S(t)=\sum_{m=0}^{\log_q t-1} \frac{(q-1)q^mE(q^m)}{[1+w(1-\eta)(q-1)q^m]P(q^{m})}, 
\end{equation}
$c\in\C$ is an arbitrary constant, and the function $E$ satisfies $|E(t)|\le\varepsilon$ for all $t\in\T$.
\end{lemma}

\begin{proof}
Suppose a function $\phi:\T\rightarrow\C$ satisfies the inequality \eqref{phiepineq}. Then, 
\[ D_q \phi(t) - w \left\langle \phi(t) \right\rangle_{\eta} = E(t) \]
for some function $E$ that satisfies $|E(t)|\le\varepsilon$ for all $t\in\T$. Expanding this out and solving for $\phi(qt)$, we have
\[ \phi(qt) = \left(\frac{1+w(1-\eta)(q-1)t}{1-w\eta(q-1)t}\right)\phi(t) + \frac{(q-1)tE(t)}{1-w\eta(q-1)t}, \quad t\in\T. \]
Solving this non-homogeneous recurrence relation for $\phi(t)$, we get $\phi=PS+cP$, where $P$ and $S$ are given in \eqref{phiform}, and $c\in\C$ is an arbitrary constant. Conversely, let $P$ be given as in \eqref{phiform}. By Lemma \ref{GeneSol}, $x=cP$ is the form of the general solution of \eqref{maineq}.
Let $\phi$ be of the form $PS+cP$, for $P$ and $S$ as in \eqref{phiform}. Then, as the operators are linear,
\begin{eqnarray*}
 \lefteqn{D_q \phi(t) - w \left\langle \phi(t) \right\rangle_{\eta}} \\
  &=&  D_q [P(t)S(t)+cP(t)] - w \left\langle P(t)S(t)+cP(t) \right\rangle_{\eta} \\
  &=& \frac{P(qt)S(qt)-P(t)S(t)}{(q-1)t} - w\left[\eta P(qt)S(qt) + (1-\eta)P(t)S(t)\right]+0 \\
  &=& \left(\frac{1-w\eta(q-1)t}{(q-1)t}\right)P(qt)S(qt) - \left(\frac{1+w(1-\eta)(q-1)t}{(q-1)t}\right)P(t)S(t) \\
  &=&  \left(\frac{1-w\eta(q-1)t}{(q-1)t}\right)\left(\frac{1+w(1-\eta)(q-1)t}{1-w\eta(q-1)t}\right)P(t)S(qt) \\
	& & - \left(\frac{1+w(1-\eta)(q-1)t}{(q-1)t}\right)P(t)S(t) \\		
  &=& \left(\frac{1+w(1-\eta)(q-1)t}{(q-1)t}\right)P(t)[S(qt) - S(t)] \\
	&=& \left(\frac{1+w(1-\eta)(q-1)t}{(q-1)t}\right)P(t)\left[\frac{(q-1)tE(t)}{\left(1+w(1-\eta)(q-1)t\right)P(t)}\right] \\
	&=& E(t).
\end{eqnarray*}
As $E$ satisfies $|E(t)|\le\varepsilon$ for all $t\in\T$, the function $\phi$ given in \eqref{phiform} satisfies \eqref{phiepineq}, completing the proof.
\end{proof}

The following lemma shows that \eqref{maineq} is unstable in the Hyers--Ulam sense when $w=0$.

% Lemma: No HUS %

\begin{lemma}\label{noHUS}
Let $q>1$ and $\eta\in\left[0,\frac{1}{2}\right)$. If $w=0$, then \eqref{maineq} is not Hyers--Ulam stable. 
\end{lemma}

\begin{proof}
Since $w=0$, $w$ satisfies \eqref{wrestriction}.
Let an arbitrary $\varepsilon>0$ be given, and for $t\in\T$ let $\phi$ have the form \eqref{phiform} with $w=0$ and $E(t)=\varepsilon P(t)$.
As $P$ has the form \eqref{generalsolformx} with $c=1$, and
\[ \left|\frac{1+w(1-\eta)(q-1)q^k}{1-w\eta(q-1)q^k}\right| = 1 \]
for all $k$, we have that $P$ is a solution of \eqref{maineq} with $P(t)\equiv 1$ for all $t\in\T$, and thus $|E(t)|=\varepsilon$ in this instance. Then, as in the proof of Lemma \ref{VofP} with $w=0$ and $E(t)=\varepsilon P(t)$, it follows that
\begin{eqnarray*}
 D_q \phi(t) = E(t),
\end{eqnarray*}
and thus $\phi$ satisfies the equality
\[ \left|D_q \phi(t)\right| = \varepsilon \]
for all $t\in\T$. If $w=0$, then 
$$ S(t)=\varepsilon(q-1)\sum_{m=0}^{\log_q t-1} q^m $$
diverges as $t\rightarrow\infty$.
Since \eqref{generalsolformx} contains the form of the general solution of \eqref{maineq}, then
\[ |\phi(t)-x(t)| = |P(t)S(t)-cP(t)| = |S(t)-c| \rightarrow \infty \]
as $t\rightarrow\infty$ for $t\in\T$ and for any $c\in\C$. In this case, by definition \eqref{maineq} lacks HUS on $\T$. 
\end{proof}

% Lemma: Convergence of S %

\begin{lemma}\label{ConvS}
Let $\varepsilon>0$, $q>1$, $\eta\in\left[0,\frac{1}{2}\right)$, and let $w\in\C\backslash\{0\}$ satisfy \eqref{wrestriction}. Let $P$ and $S$ be the functions defined by \eqref{phiform}. Suppose that the function $E$ satisfies $|E(t)|\le\varepsilon$ for all $t\in\T$. Then $\lim_{t\to\infty}S(t)$ exists, $\lim_{t\to\infty}|P(t)|=\infty$, and 
\begin{equation}\label{1overw}
 P(t)\sum_{m=\log_q t}^{\infty} \frac{(q-1)q^m}{[1+w(1-\eta)(q-1)q^m]P(q^{m})} \equiv \frac{1}{w}
\end{equation}
for all $t\in\T$. 
\end{lemma}

\begin{proof}
First, we prove the $\lim_{t\to\infty}S(t)$ exists. We only have to prove that the infinite series
\[ \sum_{m=0}^{\infty} \frac{(q-1)q^mE(q^m)}{[1+w(1-\eta)(q-1)q^m]P(q^{m})} \]
converges absolutely. Let
\begin{equation}
 a_m := \left|\frac{(q-1)q^m}{[1+w(1-\eta)(q-1)q^m]P(q^{m})}\right|
\end{equation}
for $m\in \N_0$. Then, we have
\[ \frac{a_{m+1}}{a_m} = q\left|\frac{P(q^{m})}{P(q^{m+1})}\frac{1+w(1-\eta)(q-1)q^m}{1+w(1-\eta)(q-1)q^{m+1}}\right| 
= q\left|\frac{1-w\eta(q-1)q^m}{1+w(1-\eta)(q-1)q^{m+1}}\right|. \]
Since $w\neq0$ and $\eta\in\left[0,\frac{1}{2}\right)$, we get
\[ \lim_{m\to\infty}\frac{a_{m+1}}{a_m}= \lim_{m\to\infty}q\left|\frac{1-w\eta(q-1)q^m}{1+w(1-\eta)(q-1)q^{m+1}}\right| = \frac{\eta}{1-\eta}<1. \]
Thus, by using the D'Alembert criterion (ratio test), we see that the infinite series $\sum_{m=0}^{\infty} a_m$ converges absolutely. Since
\[ \left|\sum_{m=0}^{\infty} \frac{(q-1)q^mE(q^m)}{[1+w(1-\eta)(q-1)q^m]P(q^{m})}\right| \le \varepsilon\sum_{m=0}^{\infty} a_m < \infty \]
holds, we conclude that $\sum_{m=0}^{\infty} \frac{(q-1)q^mE(q^m)}{[1+w(1-\eta)(q-1)q^m]P(q^{m})}$ converges absolutely. That is, $\lim_{t\to\infty}S(t)$ exists. 

For $\eta\in\left[0,\frac{1}{2}\right)$ and $w\ne 0$,
\[ \lim_{k\rightarrow\infty}\left|\frac{1+w(1-\eta)(q-1)q^k}{1-w\eta(q-1)q^k}\right| = \begin{cases} \frac{1-\eta}{\eta} > 1 :& \eta\in\left(0,\frac{1}{2}\right), \\ \infty :& \eta=0. \end{cases} \] 
Thus, we get $\lim_{t\to\infty}|P(t)|=\infty$. 

Next, we will prove
\begin{equation}\label{kappadef}
 \kappa(t) := \left|P(t)\sum_{m=\log_q t}^{\infty} \frac{(q-1)q^m}{[1+w(1-\eta)(q-1)q^m]P(q^{m})}\right| 
\end{equation}
is bounded above on $\T$. Since
\[ \lim_{t\to\infty}\left|\frac{1-w\eta(q-1)q^k t}{1+w(1-\eta)(q-1)q^k t}\right| = \frac{\eta}{1-\eta}<1,\quad k\in\N_0 \]
and
\[ \lim_{t\to\infty}\frac{1}{\left|\frac{1}{(q-1)q^m t}+w(1-\eta)\right|} = \frac{1}{|w(1-\eta)|}>0,\quad m\in\N_0 \]
hold, for any $\eta\in\left[0,\frac{1}{2}\right)$, there exist $T\in\T$, $\frac{\eta}{1-\eta}<\delta<1$ and $\gamma > \frac{1}{|w(1-\eta)|}$ such that
\[ \left|\frac{1-w\eta(q-1)q^k t}{1+w(1-\eta)(q-1)q^k t}\right| \le \delta \quad\text{and}\quad \frac{1}{\left|\frac{1}{(q-1)q^m t}+w(1-\eta)\right|} \le \gamma \]
for $t\ge T$, $t\in\T$, $k\in\N_0$ and $m\in\N_0$. Using this and the inequality
\begin{eqnarray*}
 \kappa(t) &\le& |P(t)|\sum_{m=\log_q t}^{\infty} \frac{1}{\left|\frac{1}{(q-1)q^m}+w(1-\eta)\right||P(q^{m})|}\\
  &=& |P(t)|\Bigg(\frac{1}{\left|\frac{1}{(q-1)t}+w(1-\eta)\right||P(t)|}+\frac{1}{\left|\frac{1}{(q-1)qt}+w(1-\eta)\right||P(qt)|}\\
  & &+\frac{1}{\left|\frac{1}{(q-1)q^2t}+w(1-\eta)\right||P(q^2t)|}+\cdots \Bigg)\\
  &=& \sum_{m=0}^{\infty}\left(\frac{1}{\left|\frac{1}{(q-1)q^m t}+w(1-\eta)\right|}\prod_{k=0}^{m-1}\left|\frac{1-w\eta(q-1)q^k t}{1+w(1-\eta)(q-1)q^k t}\right| \right)
\end{eqnarray*}
for $t\in\T$, we see that
\[ \kappa(t) \le \gamma\sum_{m=0}^{\infty}\left(\prod_{k=0}^{m-1}\delta \right) = \gamma\sum_{m=0}^{\infty}\delta^m = \frac{\gamma}{1-\delta} \]
for $t\ge T$ and $t\in\T$. This means that $\kappa(t)$ is bounded above on $\T$. 

Last, we will prove
\[ \psi(t) := P(t)\sum_{m=\log_q t}^{\infty} \frac{(q-1)q^m}{[1+w(1-\eta)(q-1)q^m]P(q^{m})} \equiv \frac{1}{w} \]
for all $t\in\T$, where $|\psi|=\kappa$ for $\kappa$ defined earlier in \eqref{kappadef}. First, note that
\[ P(qt) = \frac{1+w(1-\eta)(q-1)t}{1-w\eta(q-1)t}P(t) \]
and
\begin{eqnarray*}
 \sum_{m=\log_q t+1}^{\infty} \frac{(q-1)q^m}{\left[1+w(1-\eta)(q-1)q^m\right]P(q^{m})} 
&=& -\frac{(q-1)t}{\left[1+w(1-\eta)(q-1)t\right]P(t)} + \frac{\psi(t)}{P(t)},
\end{eqnarray*}
yielding
\begin{eqnarray*}
 \psi(qt) &=& P(qt)\sum_{m=\log_q t+1}^{\infty} \frac{(q-1)q^m}{\left[1+w(1-\eta)(q-1)q^m\right]P(q^{m})} \\
  &=& \frac{-(q-1)t}{1-w\eta(q-1)t} +\frac{1+w(1-\eta)(q-1)t}{1-w\eta(q-1)t}\psi(t).
\end{eqnarray*}
Consequently, we have
\begin{eqnarray*}
 D_q\psi(t) &=& \frac{\psi(qt)-\psi(t)}{(q-1)t} \\
  &=& \frac{1}{(q-1)t}\left[\frac{-(q-1)t}{1-w\eta(q-1)t} + \frac{1+w(1-\eta)(q-1)t}{1-w\eta(q-1)t}\psi(t)-\psi(t) \right]\\
  &=& \frac{1}{(q-1)t}\left[\frac{-(q-1)t}{1-w\eta(q-1)t} + \frac{w(q-1)t\psi(t)}{1-w\eta(q-1)t}\right]\\
  &=& \frac{w\psi(t)-1}{1-w\eta(q-1)t}.
\end{eqnarray*}
Moreover, we see that
\begin{eqnarray*}
 w \left\langle \psi(t) \right\rangle_{\eta} &=& w\eta \psi(qt) +w(1-\eta)\psi(t) \\
 &=& \frac{w\psi(t)-w\eta(q-1)t}{1-w\eta(q-1)t},
\end{eqnarray*}
so that 
\[ D_q\psi(t) - w \left\langle \psi(t) \right\rangle_{\eta} = \frac{w\psi(t)-1}{1-w\eta(q-1)t} - \frac{w\psi(t)-w\eta(q-1)t}{1-w\eta(q-1)t}=-1. \]
This means that $\psi$ is a solution of $D_q\psi(t)-w\left\langle\psi(t)\right\rangle_{\eta}=-1$. Since $\frac{1}{w}$ is also a 
(constant) solution of this equation, by Lemma \ref{GeneSol} and the superposition principle, we can rewrite $\psi$ as
\[ \psi(t) = cP(t)+\frac{1}{w} \]
for some $c\in\C$. Suppose that $c\ne 0$. Using the earlier part of this proof, we see that 
$$\lim_{t\to\infty}|\psi(t)| = \lim_{t\to\infty}\kappa(t) = \infty, $$ 
which contradicts the boundedness of $\kappa=|\psi|$. Thus, $c=0$, yielding $\psi(t) \equiv \frac{1}{w}$. 
This completes the proof.
\end{proof}

% Theorem: main theorem %

\begin{theorem}\label{mainthm}
Let $q>1$, $\eta\in\left[0,\frac{1}{2}\right)$, and let $w\in\C\backslash\{0\}$ satisfy \eqref{wrestriction}. Let $P$ be the function defined by \eqref{phiform}. Let $\varepsilon>0$ be a fixed arbitrary constant, and let $\phi$ be a function on $\T$ satisfying the inequality
\[ |D_q\phi(t) - w \left\langle \phi(t)\right\rangle_{\eta}|\le\varepsilon, \qquad t\in\T. \]
Then, $\displaystyle\lim_{t\to\infty}\frac{\phi(t)}{P(t)}$ exists, and the function $x$ given by 
\[ x(t):=\left(\lim_{t\to\infty}\frac{\phi(t)}{P(t)}\right)P(t) \] 
is the unique solution of \eqref{maineq} with 
$$|\phi(t)-x(t)| \le \frac{\varepsilon}{|w|} $$ 
for all $t\in\T$.
\end{theorem}

\begin{proof}
Let $w\in\C\backslash\{0\}$ satisfy condition \eqref{wrestriction}. Let $\varepsilon>0$ be given. Suppose that $\left|D_q\phi(t) - w \left\langle \phi(t)\right\rangle_{\eta}\right|\le\varepsilon$ for all $t\in\T$. Let $E(t) := D_q\phi(t) - w \left\langle \phi(t)\right\rangle_{\eta}$. Then
\[ D_q\phi(t) - w \left\langle \phi(t)\right\rangle_{\eta} = E(t), \quad |E(t)|\le\varepsilon \]
for all $t\in\T$. Lemma \ref{VofP} implies that $P(t)S(t)$ is a solution of this equation, where $P(t)$ and $S(t)$ are given in \eqref{phiform}. Since $P(t)$ is a solution of \eqref{maineq} by Lemma \ref{GeneSol}, we can write $\phi:\T\rightarrow\C$ as
\[ \phi(t) := \phi_0P(t)+P(t)S(t), \]
where $\phi_0\in\C$ is a suitable constant, by Lemma \ref{VofP}. Note here that $P(t)\neq 0$ for all $t\in\T$ since $w\in\C\backslash\{0\}$ satisfying \eqref{wrestriction}. From Lemma \ref{ConvS} the $\lim_{t\to\infty}S(t)$ exists, so that 
\[ x_0:= \lim_{t\to\infty}\frac{\phi(t)}{P(t)} = \phi_0+\lim_{t\to\infty}S(t) \]
exists. Moreover, we can rewrite $\phi$ as
\begin{eqnarray*}
 \phi(t) &=& \left(\phi_0+\lim_{t\to\infty}S(t)\right)P(t)+P(t)\left(S(t)-\lim_{t\to\infty}S(t)\right)\\
  &=& x_0P(t)-P(t)\sum_{m=\log_q t}^{\infty} \frac{(q-1)q^mE(q^m)}{[1+w(1-\eta)(q-1)q^m]P(q^{m})}
\end{eqnarray*}
for all $t\in\T$. Now we consider the function
\[ x(t):=x_0P(t). \]
Then, $x$ is a solution of \eqref{maineq} by Lemma \ref{GeneSol}. Hence, we have
\begin{eqnarray*}
 |\phi(t)-x(t)| &=& \left|-P(t)\sum_{m=\log_q t}^{\infty} \frac{(q-1)q^mE(q^m)}{[1+w(1-\eta)(q-1)q^m]P(q^{m})}\right|\\
  &\le& \varepsilon \left|P(t)\sum_{m=\log_q t}^{\infty} \frac{(q-1)q^m}{[1+w(1-\eta)(q-1)q^m]P(q^{m})}\right| \\
	&=& \frac{\varepsilon}{|w|}
\end{eqnarray*}
for all $t\in\T$ from Lemma \ref{ConvS}. As a result, \eqref{maineq} has Hyers--Ulam stability with HUS constant
\[ K:=\frac{1}{|w|}. \]

Next, we will show that $x(t)=x_0P(t)$ is the unique solution of \eqref{maineq} with $|\phi(t)-x(t)|\le K\varepsilon$ for all $t\in\T$. Suppose that 
\[ \left|D_q\phi(t) - w \left\langle \phi(t)\right\rangle_{\eta}\right|\le\varepsilon \]
for all $t\in\T$, and that there exist two different solutions $z_1$ and $z_2$ of \eqref{maineq} with $|\phi(t)-z_l(t)|\le K\varepsilon$ for $t\in\T$ and $l=1$, $2$. Then, we can rewrite $z_l$ as
\[ z_l(t) = c_lP(t), \quad c_1 \neq c_2. \]
Henc,e we see that
\[ |c_2-c_1||P(t)| = |z_2(t)-z_1(t)| \le |z_2(t)-\phi(t)|+|\phi(t)-z_1(t)| \le 2 K\varepsilon \]
for all $t\in\T$. Using Lemma \ref{ConvS}, we have $\lim_{t\to\infty}|P(t)|=\infty$; this is a contradiction. Thus, $x(t)=x_0P(t)$ is the unique solution. 
\end{proof}

Following the main theorem above, we can affirm that the HUS constant of $K=\frac{1}{|w|}$ is the best possible constant for 
$\eta\in\left[0,\frac{1}{2}\right)$ and $w\in\C\backslash\{0\}$ satisfying \eqref{wrestriction}. 

% Theorem: HUS %

\begin{theorem}\label{bestthm}
Let $q>1$, $\eta\in\left[0,\frac{1}{2}\right)$, and let $w\in\C\backslash\{0\}$ satisfy \eqref{wrestriction}. Then \eqref{maineq} has HUS with best (minimal) HUS constant $\frac{1}{|w|}$ on $\T$.
\end{theorem}

\begin{proof}
Given an arbitrary $\varepsilon>0$, set
\[ \phi(t)\equiv \frac{\varepsilon}{w}, \]
where we assume $q>1$, $\eta\in\left[0,\frac{1}{2}\right)$, and $w\in\C\backslash\{0\}$ satisfies \eqref{wrestriction}. Then, $\phi$ satisfies
\[ \left|D_q\phi(t)-w\langle\phi(t)\rangle_{\eta}\right| = |-\varepsilon| = \varepsilon, \]
and 
\[ x(t):=\left(\lim_{t\to\infty}\frac{\phi(t)}{P(t)}\right)P(t)= 0 \] 
is the unique solution of \eqref{maineq} with 
$$|\phi(t)-x(t)| \le \frac{\varepsilon}{|w|} $$ 
for all $t\in\T$, by Theorem \ref{mainthm}. This means that the best HUS constant is at least $\frac{1}{|w|}$, whereas Theorem \ref{mainthm} says it is at most $\frac{1}{|w|}$. Therefore, $\frac{1}{|w|}$ is the best possible HUS constant.
\end{proof}

%%%%%%%%%%%%%%%%%%%%%%%
% Section: \eta = 1/2 %
%%%%%%%%%%%%%%%%%%%%%%%

\section{The $\eta=\frac{1}{2}$ case leads to instability across the complex plane}

The results in the previous sections are for $\eta\in[0,\frac{1}{2})$. What happens if $\eta=\frac{1}{2}$? In stark contrast to the results contained in Lemma \ref{noHUS} and Theorem \ref{bestthm}, in which \eqref{maineq} has HUS with best HUS constant $\frac{1}{|w|}$ for all $w\in\C$ satisfying \eqref{wrestriction} except for $w=0$, equation \eqref{maineq} has no Hyers--Ulam stability for any $w\in\C$ for those $w$ that satisfy \eqref{wrestriction}, when $\eta=\frac{1}{2}$. 

% Theorem %

\begin{theorem}\label{etahalfnoHUS}
Let $\eta=\frac{1}{2}$ and $q>1$. Then for all $w\in\C$ satisfying \eqref{wrestriction}, equation \eqref{maineq} has no Hyers--Ulam stability.
\end{theorem}

\begin{proof}
Let $w\in\C$ that satisfies \eqref{wrestriction} and an arbitrary $\varepsilon>0$ be given. For $t\in\T$, let $\phi$ have the form \eqref{phiform} for this $w$, and set 
\[ E(t)=\frac{\varepsilon P(t)}{|P(t)|}, \]
where $P$ has the form \eqref{generalsolformx} and is a solution of \eqref{maineq}. Clearly, $|E(t)|=\varepsilon$ in this instance. Then, as in the proof of Lemma \ref{VofP}, it follows that
\begin{eqnarray*}
	D_q \phi(t) - w\left\langle \phi(t) \right\rangle_{\frac{1}{2}} = E(t),
\end{eqnarray*}
and thus $\phi$ satisfies the equality
\[ \left|D_q \phi(t) - w\left\langle \phi(t) \right\rangle_{\frac{1}{2}}\right| = \varepsilon \]
for all $t\in\T$. From \eqref{phiform}, we have
\[ P(t)=\prod_{k=0}^{\log_q t-1} \frac{1+\frac{w}{2}(q-1)q^k}{1-\frac{w}{2}(q-1)q^k} \qquad\text{and}\qquad S(t)=\sum_{m=0}^{\log_q t-1} \frac{\varepsilon(q-1)q^m}{[1+\frac{w}{2}(q-1)q^m]|P(q^{m})|}. \]
If $w=0$, then 
\[ P(t)=\prod_{k=0}^{\log_q t-1} 1\equiv 1 \qquad\text{and}\qquad S(t)=\sum_{m=0}^{\log_q t-1} \frac{\varepsilon(q-1)q^m}{|P(q^{m})|}=\varepsilon(q-1)\sum_{m=0}^{\log_q t-1}q^m, \]
and we see that $S(t)$ diverges to infinity as $t\rightarrow\infty$, since $q>1$. 
If $w\ne 0$, then by focusing on $P(t)$ we note that 
\[ \lim_{k\rightarrow\infty} \frac{1+\frac{w}{2}(q-1)q^k}{1-\frac{w}{2}(q-1)q^k}=-1 \]
for any $w\in\C$ satisfying \eqref{wrestriction} and for any $q>1$, and $P(t)$ converges to a two-cycle $\pm p^*$ for some $p^*\in\C\backslash\{0\}$, as $t\rightarrow\infty$. Moreover, for $S(t)$, we have
\[ \lim_{m\rightarrow\infty}\frac{\varepsilon(q-1)q^m}{[1+\frac{w}{2}(q-1)q^m]|P(q^{m})|} = \frac{2\varepsilon}{w|p^*|} \not= 0, \]
so that $S(t)$ diverges. Therefore, for any $w\in\C$ that satisfies \eqref{wrestriction}, we have that $P(t)$ is bounded away from 0, and $S(t)$ grows without bound in the limit at infinity. Since \eqref{generalsolformx} contains the form of the general solution of \eqref{maineq}, then
\[ |\phi(t)-x(t)| = |P(t)S(t)-cP(t)| = |P(t)||S(t)-c| \rightarrow \infty \]
as $t\rightarrow\infty$ for $t\in\T$ and for any $c\in\C$. In this case, by definition \eqref{maineq} lacks HUS on $\T$ for any $w\in\C$ that satisfies \eqref{wrestriction}. 
\end{proof}

\begin{example}
Let $\eta=\frac{1}{2}$ and $q>1$. If $\operatorname{Re}(w)=0$, then $w=i\beta$ for some $\beta\in\R$ and $w$ satisfies \eqref{wrestriction}.
Let an arbitrary $\varepsilon>0$ be given, and for $t\in\T$ let $\phi$ have the form \eqref{phiform} with $w=i\beta$ and $E(t)=\varepsilon P(t)/|P(t)|$.
As $P$ has the form \eqref{generalsolformx} and
\[ \left|\frac{1+\frac{i\beta}{2}(q-1)q^k}{1-\frac{i\beta}{2}(q-1)q^k}\right| = 1 \]
for all $k$, we have that $P$ is a solution of \eqref{maineq} with $|P(t)|=1$ for all $t\in\T$. If $w=10i$ and $q=2$, for example, then $P$ converges to the two-cycle $\{p^*,-p^*\}$, where $p^*=0.700975 - 0.713186 i$. Then, as in the proof of Lemma \ref{VofP} with $w=i\beta$ and $E(t)=\varepsilon P(t)/|P(t)|$, it follows that
\begin{eqnarray*}
 D_q \phi(t) - i\beta \left\langle \phi(t) \right\rangle_{\frac{1}{2}} = E(t),
\end{eqnarray*}
and thus $\phi$ satisfies the equality
\[ \left|D_q \phi(t) - i\beta \left\langle \phi(t) \right\rangle_{\frac{1}{2}}\right| = \varepsilon \]
for all $t\in\T$. If $\beta=0$, then 
$$ S(t)=\varepsilon(q-1)\sum_{m=0}^{\log_q t-1} q^m $$
diverges as $t\rightarrow\infty$; for $\beta\ne 0$,
\[ \left|\frac{q^m}{2+i\beta(q-1)q^m}\right| = \frac{q^m}{\sqrt{4+\beta^2(q-1)^2q^{2m}}} = \frac{1}{\sqrt{\frac{4}{q^{2m}}+\beta^2(q-1)^2}}, \]
which does not vanish as $t\rightarrow\infty$, and again $S(t)$ diverges at infinity.
Since \eqref{generalsolformx} contains the form of the general solution of \eqref{maineq}, then
\[ |\phi(t)-x(t)| = |P(t)S(t)-cP(t)| = |S(t)-c| \rightarrow \infty \]
as $t\rightarrow\infty$ for $t\in\T$ and for any $c\in\C$. In this case, by definition \eqref{maineq} lacks HUS on $\T$, as predicted by Theorem \ref{etahalfnoHUS}. 
\end{example}

\begin{example}
Let $\eta=\frac{1}{2}$. For $w=1-2i$ and $q=2.5$, the function $P$ converges to the two-cycle $\pm p^*$, for $p^*=-0.35346 + 2.11351 i$. For $w=-2+5i$ and $q=1.5$, $P$ converges to the two-cycle $\pm p^*$, for $p^*=-0.170672 + 0.183965 i$. If $w=-3$ and $q=2$, then $P$ converges to $\pm 0.0511582$. Finally, if $w=\pi$ and $q=1.8$, then $P$ converges to $\pm 69.4908$. In all of these cases, the sum $S$ diverges, and \eqref{maineq} is not Hyers-Ulam stable by Theorem \ref{etahalfnoHUS}.
\end{example}

% Future Directions %

\section{Future Directions}

A Hyers--Ulam stability analysis for the first-order quantum equation is presented here for the first time, yielding a sharp (minimal) Hyers--Ulam constant. In the future, one would like to build on these results for first-order quantum equations to get results for second-order quantum equations with constant coefficients.

\section*{Acknowledgements}
The authors would like to thank the referees for reading carefully and giving valuable comments to improve the quality of the paper. The second author was supported by JSPS KAKENHI Grant Number JP20K03668.

%%%%%%%%%%%%%%%%
% Bibliography %
%%%%%%%%%%%%%%%%

\end{document}